\documentclass[12pt,a4paper]{amsart}
\usepackage{amsfonts,amsmath,amssymb,amscd}
\usepackage{fancybox}
\usepackage[latin1]{inputenc}
\usepackage[usenames]{color}
\usepackage{graphicx}
\usepackage{enumerate}
\usepackage{subfig}
\usepackage{float}
\usepackage{fullpage}
\usepackage{epsfig,amssymb}

\theoremstyle{plain}
\newtheorem{thm}{Theorem}[section]
\newtheorem{cor}[thm]{Corollary}
\newtheorem{lem}[thm]{Lemma}
\newtheorem{prop}[thm]{Proposition}

\theoremstyle{definition}

\newtheorem{rem}[thm]{Remark}

\begin{document}
	
\title{Closed Biconservative Hypersurfaces in Spheres}

\author{S. Montaldo}
\address{Universit\`a degli Studi di Cagliari\\
Dipartimento di Matematica e Informatica\\
Via Ospedale 72\\
09124 Cagliari, Italia}
\email{montaldo@unica.it}

\author{C. Oniciuc}
\address{``Alexandru Ioan Cuza" University of Iasi\\
Faculty of Mathematics\\
Blvd. Carol I no. 11\\
Iasi 700506, Romania}
\email{oniciucc@uaic.ro}

\author{A. P\'ampano}
\address{Department of Mathematics and Statistics\\
Texas Tech University\\
79409, Lubbock, TX, USA}
\email{alvaro.pampano@ttu.edu}
\date{\today}

\begin{abstract} We characterise the profile curves of non-CMC biconservative rotational hypersurfaces of space forms $N^n(\rho)$ as $p$-elastic curves, for a suitable rational number $p\in[1/4,1)$ which depends on the dimension $n$ of the ambient space. Analysing the closure conditions of these $p$-elastic curves, we prove the existence of a discrete biparametric family of non-CMC closed (i.e., compact without boundary) biconservative hypersurfaces in $\mathbb{S}^n(\rho)$. None of these hypersurfaces can be embedded in $\mathbb{S}^n(\rho)$.
\end{abstract}

\thanks{}
\subjclass[2010]{}

\maketitle

\section{Introduction}

An oriented hypersurface $M^{n-1}$ immersed in a $n$-dimensional Riemannian space form $N^n(\rho)$ is called \emph{biconservative} if
\begin{equation}\label{eq:biconservative-general-formula-sf}
2S_\eta\left({\rm grad} H\right)+\left(n-1\right)H {\rm grad} H=0\,,
\end{equation}
where $\eta$ is a unit normal vector field, $S_\eta$ is the shape operator and $H={\rm trace} S_\eta/(n-1)$ is the mean curvature function. Biconservative immersions were first introduced in \cite{CMOP} as those immersions with conservative stress energy-tensor associated to the bienergy. Their study became soon quite intense and we refer the reader to the recent survey \cite{FO} and the references therein.

When $H$ is constant \eqref{eq:biconservative-general-formula-sf} trivially holds, thus the class of biconservative hypersurfaces includes that of constant mean curvature (CMC) hypersurfaces and, consequently, it can be considered as a geometric generalisation of the latter. To understand how much this class of hypersurfaces is larger than that of CMC ones, the main interest is to study non-CMC biconservative hypersurfaces.

When $n=3$, biconservative surfaces of $3$-dimensional Riemannian space forms $N^3(\rho)$ were locally classified in \cite{CMOP} and globally in \cite{Simona}. More precisely, it was proved that non-CMC biconservative surfaces must be rotational and a relation between the mean and Gaussian curvatures was stated. Later on, in \cite{MP}, taking advantage of these results, the authors gave a variational characterisation of profile curves of non-CMC biconservative surfaces. They showed that the profile curve is a $p$-elastic curve for $p=1/4$. We point out here that the variational problem associated to $p$-elastic curves is, a priori, unrelated to the energies usually studied in the theory of harmonic and biharmonic submanifolds. To the contrary, it is a classical problem whose energy can be seen as an extension of the bending energy for curves which gives rise to classical elastic curves. In a letter of 1738 from D. Bernoulli to L. Euler \cite{Tr}, both elastic and $p$-elastic curves were introduced and treated simultaneously.

In the same paper (\cite{MP}), using this variational characterisation and employing techniques arising from curvature energy problems, it was shown the existence of a discrete biparametric family of closed non-CMC biconservative surfaces in $\mathbb{S}^3(\rho)$, none of which were embedded. The question that naturally arises here is whether or not closed non-CMC biconservative hypersurfaces exist in higher dimension. This question was eloquently described and motivated as an open problem in the recent report \cite{FO} (the same open problem appears in \cite{FLO}). In the current paper we give an affirmative answer to this question and find a discrete biparametric family of closed non-CMC biconservative hypersurfaces in $\mathbb{S}^n(\rho)$, for every arbitrary dimension $n\geq 3$.

We first show, using results of \cite{DC-D}, that for $n\geq 3$ a non-CMC biconservative hypersurface with exactly two distinct principal curvatures of a space form $N^3(\rho)$  is rotational and, for this class of hypersurfaces, we  extend the variational characterisation for their profile curves. More precisely, we prove in Theorem \ref{char} that a non-CMC rotational hypersurface of $N^n(\rho)$ is biconservative if and only if its profile curve is a $p$-elastic curve where $p=(n-2)/(n+1)$ depends on the dimension of the ambient space.

Next, by analysing the phase space of the Euler-Lagrange equation for $p$-elastic curves, we conclude in Proposition \ref{periodic} that if closed $p$-elastic curves exist they must lie in $\mathbb{S}^2(\rho)$, in which case all $p$-elastic curves have periodic curvature, provided that they are defined on their maximal domain. In Theorem \ref{closed}, we prove the existence of a discrete biparametric family of closed $p$-elastic curves in $\mathbb{S}^2(\rho)$, independently of the value of $p$, i.e., regardless of the dimension of the ambient space $\mathbb{S}^n(\rho)$. In Figures \ref{Figure1} and \ref{Figure2} we show the two simplest closed $p$-elastic curves for several choices of the dimension $n$.

Finally, in Section \ref{5}, we rewrite these results in terms of the corresponding rotational hypersurfaces obtaining, in Theorem \ref{claimed}, the claimed result about the existence of a discrete biparametric family of non-CMC closed biconservative hypersurfaces in $\mathbb{S}^n(\rho)$.

\section{Rotational Biconservative Hypersurfaces}

Harmonic maps between Riemannian manifolds are the critical points of the energy functional. In \cite{ES}, Eells and Sampson suggested to study biharmonic maps, which are the critical points of the bienergy functional. The first variation formula of the bienergy was derived by Jiang (\cite{J1}), who also obtained the associated Euler-Lagrange equation. For details see \cite{FO} and references therein.

Denote by $\varphi:M^m\rightarrow N^n$ an isometric immersion between Riemannian manifolds of dimensions $m<n$. Then, the decomposition of the Euler-Lagrange equation associated to the bienergy with respect to its normal and tangent components was obtained with contributions of \cite{BMO,C,LM,O}. If $m=n-1$, we will say that $M^{n-1}$ is a \emph{hypersurface} of $N^n$ (throughout this paper all hypersurfaces are assumed to be oriented). In particular, this decomposition for hypersurfaces can be summarised in the following theorem.

\begin{thm}\label{decomposition} Let $\varphi:M^{n-1}\rightarrow N^n$ be an isometric immersion with mean curvature function $H$. Then, the normal and tangential components of the Euler-Lagrange equation associated to the bienergy are, respectively,
\begin{eqnarray*}
\Delta H+H \lvert S_\eta\rvert^2+H{\rm Ric}(\eta,\eta)&=&0\,,\\
2S_\eta\left({\rm grad} H\right)+(n-1)H {\rm grad} H-2H {\rm Ric}(\eta)^{\top}&=&0\,,
\end{eqnarray*}
where $\eta$ is the unit normal vector field, $S_\eta$ is the shape operator and ${\rm Ric}$ denotes the Ricci curvature of $N^n$. Here, the symbol $\Delta$ represents the rough Laplacian.
\end{thm}

In order to interpret the tangential component of the Euler-Lagrange equation associated to the bienergy, we recall that, as described by Hilbert in \cite{H}, the stress-energy tensor associated with a variational problem is a symmetric $2$-covariant tensor $\mathcal{S}$ which is conservative at critical points, that is, ${\rm div}\mathcal{S}=0$ holds. For instance, for the energy functional the stress-energy tensor was studied in \cite{BE} and \cite{S}. 

If $\mathcal{S}$ is the stress-energy tensor for the bienergy, from the computations of \cite{J} and \cite{LMO}, it follows that for an isometric immersion $\varphi:M^{n-1}\rightarrow N^n$ the condition ${\rm div}\mathcal{S}=0$ reduces to
\begin{equation}\label{eq:biconservative-general-formula}
2S_\eta\left({\rm grad} H\right)+(n-1)H {\rm grad} H-2H {\rm Ric}(\eta)^{\top}=0\,,
\end{equation}
which is exactly the tangential component of the Euler-Lagrange equation associated to the bienergy. We then say that an isometric immersion is \emph{biconservative} if the stress-energy tensor for the bienergy is conservative, that is if  Equation~\eqref{eq:biconservative-general-formula} holds. A hypersurface $M^{n-1}$ immersed in this way is usually called a \emph{biconservative hypersurface}.

Let $N^n(\rho)$ be a $n$-dimensional Riemannian space form of constant sectional curvature $\rho$. If $\rho=0$, $N^n(\rho)=\mathbb{R}^n$ denotes the \emph{Euclidean space} of dimension $n$; when $\rho>0$, we have the $n$-dimensional \emph{round sphere} $\mathbb{S}^n(\rho)$; and, finally, for negative $\rho$ we recover the $n$-dimensional \emph{hyperbolic space}, simply denoted by $\mathbb{H}^n(\rho)$. In these particular ambient spaces, the tangential part of the Ricci curvature vanishes and, hence, a hypersurface immersed in $N^n(\rho)$ is biconservative if and only if Equation \eqref{eq:biconservative-general-formula-sf} holds.
As mentioned in the introduction, trivial examples of biconservative hypersurfaces of space forms are provided by constant mean curvature (CMC) hypersurfaces. Thus, our main interest is to investigate  non-CMC biconservative hypersurfaces. 

Let $M^{n-1}$ be a hypersurface with principal curvatures $\kappa_i$, $i=1,...,n-1$, i.e., the eigenvalue functions of the shape operator $S_\eta$. The functions $\kappa_{i}$ are continuous on $M^{n-1}$ for all $i=1,...,n-1$. The set of points where the number of distinct principal curvatures is locally constant is a set $M_{A}$ that is open and dense in $M^{n-1}$. On a non-empty connected component of $M_{A}$, which is open in $M_{A}$ and so in $M^{n-1}$, the number of distinct principal curvatures is constant. Thus, the multiplicities of the distinct principal curvatures are also constant, and so, on that connected component, the principal curvatures $\kappa_{i}$ are smooth and $S_\eta$ is (smoothly) locally diagonalizable (see \cite{Nomizu,Ryan1,Ryan2}). For a hypersurface $M^{n-1}$ with nowhere zero ${\rm grad} H$, the corresponding $M_A$ has no connected component made up only by umbilical points (that is, there, the number of distinct principal curvatures is constant one). In this setting, from \eqref{eq:biconservative-general-formula-sf}, it follows that a biconservative hypersurface $M^{n-1}$ with nowhere zero ${\rm grad}H$ has a principal curvature $\kappa_1=-(n-1)H/2$ of multiplicity one on each connected component of $M_A$ and, hence,
\begin{equation}\label{lw}
3\kappa_1+\kappa_2+\dots+\kappa_{n-1}=0
\end{equation}
holds.

Let $M^{n-1}$ be a (spherical) \emph{rotational} hypersurface of $N^n(\rho)$, that is, a hypersurface invariant by the orthogonal group $O(n-1)$ considered as a subgroup of isometries of the ambient space $N^n(\rho)$. The orbit of a point in $N^n(\rho)$ under the action $O(n-1)$ is a $(n-2)$-dimensional sphere. Consequently, the hypersurface $M^{n-1}$ can be described as the evolution of an arc length parametrized curve $\gamma(s)$ in $N^2(\rho)$ under the action of $O(n-1)$. We call $\gamma(s)=(x_1(s),x_2(s),x_3(s))\subset N^2(\rho)\subset\mathbb{R}^3$ the \emph{profile curve} of $M^{n-1}$.

From Proposition 3.2 of \cite{DC-D} we have that a rotational hypersurface has, at most, two distinct principal curvatures:
\begin{equation}\label{k1}
\kappa_1=\mu=\frac{x_1''+\rho x_1}{\sqrt{1-\rho x_1^2-(x_1')^2}}\,,
\end{equation}
which corresponds to minus the signed curvature $\kappa(s)$ of the profile curve $\gamma(s)$ (its associated eigenvector is parallel to ${\rm grad}H$); and
\begin{equation}\label{k2}
\kappa_2=\dots=\kappa_{n-1}=\lambda=-\frac{\sqrt{1-\rho x_1^2-(x_1')^2}}{x_1}\,,
\end{equation}
which has multiplicity, at least, $n-2$. Here, we are denoting by $\left(\,\right)'$ the derivative with respect to the arc length parameter $s$ of the profile curve $\gamma(s)$.

If a non-CMC biconservative hypersurface $M^{n-1}$ of a space form $N^n(\rho)$ has exactly two distinct principal curvatures then, from \eqref{lw}, the principal curvatures satisfy
\[
\kappa_2=\dots=\kappa_{n-1}\,,\quad \kappa_1=\frac{2-n}{3} k_2
\]
and, if $n\geq 4$, a result of Do Carmo--Dajczer \cite[Theorem 4.2]{DC-D} ensures that  $M^{n-1}$ is a rotational hypersurface. If $n=3$, Theorem~4.2 of \cite{DC-D} cannot be applied. Nevertheless, non-CMC biconservative surfaces are still rotational, as proved in \cite{CMOP}.
%

We sum up this in the following proposition.

\begin{prop} Let $M^{n-1}$ be a hypersurface of a space form $N^n(\rho)$ with nowhere zero ${\rm grad}H$ and with exactly two distinct principal curvatures everywhere. Then, $M^{n-1}$ is biconservative if and only if it is rotational and the principal curvatures $\mu$ (of multiplicity one) and $\lambda$ (of multiplicity $n-2$) satisfy
$$3\mu+\left(n-2\right)\lambda=0\,.$$
Recall that $\kappa=-\mu$ is the signed curvature of the profile curve $\gamma\subset N^2(\rho)$ of the hypersurface $M^{n-1}$.
\end{prop}


\section{Variational Characterisation of Profile Curves}

In this section, we will characterise the profile curves of non-CMC rotational biconservative hypersurfaces $M^{n-1}$ as $p$-elastic curves, for a suitable rational number $p\equiv p(n)$ depending on the dimension of the ambient space $N^n(\rho)$. For this purpose, we briefly recall here the general theory for curvature energies (for more details see, for instance, \cite{AP} and references therein).

Consider the curvature energy functional 
$$\mathbf{\Theta}(\gamma)=\int_\gamma P(\kappa)\,,$$
where $P(\kappa)$ is a smooth function defined on an adequate domain, $\kappa$ is the curvature of $\gamma$ and $\mathbf{\Theta}$ is acting on the space of curves immersed in $N^2(\rho)$.

\begin{rem} The curvature energy functional $\mathbf{\Theta}$ acts on the space of curves isometrically immersed in $N^2(\rho)$ with the pullback metric, i.e.,
$$\mathbf{\Theta}(\gamma)=\int_\gamma P(\kappa)=\int_0^L P\left(\kappa(s)\right)ds=\int_0^1 P\left(\kappa(t)\right) v(t)\,dt\,,$$
where $L$ is the length of $\gamma$ and $v(t)=\lVert d\gamma(t)/dt\rVert$. In other words, the induced metric changes throughout the variation curves and, hence, $\mathbf{\Theta}$ is not the energy of the map, nor an extended notion. It represents a generalization of the classical bending energy for curves.
\end{rem}

Regardless of the boundary conditions, a curve critical for $\mathbf{\Theta}$ satisfies the following \emph{Euler-Lagrange equation}
\begin{equation}\label{EL}
\dot{P}_{ss}+\dot{P}\left(\kappa^2+\rho\right)-\kappa P=0\,.
\end{equation}
Here, $\dot{P}$ is the derivative of $P$ with respect to $\kappa$ and $s$ represents the arc length parameter of the curve. We are denoting with subindexes the derivatives with respect to the arc length parameter $s$. In particular, if the curvature $\kappa$ of the critical curve is nonconstant, equation \eqref{EL} can be integrated once. Indeed, by multiplying it by $\dot{P}_s$, we obtain an exact differential and so
\begin{equation}\label{fi}
\dot{P}_s^2+\left(\kappa\dot{P}-P\right)^2+\rho\dot{P}^2=d
\end{equation}
holds for a suitable constant of integration $d$. Observe that, in order to obtain \eqref{fi}, we are assuming $\dot{P}_s\neq 0$ which, since $\kappa$ is nonconstant means that $P(\kappa)\neq a\kappa+b$ for constants $a$, $b\in\mathbb{R}$. In order to obtain closed orbits, we are going to restrict the constant of integration $d$ to be positive (for details, cf. the explanations in \cite{DC-D} and \cite{AP}). This is always the case when $\rho\geq 0$ holds, i.e., in the Euclidean plane $\mathbb{R}^2$ and in the round sphere $\mathbb{S}^2(\rho)$.

Curves whose curvature is a solution of \eqref{fi} for $d>0$ can be characterised as follows.

\begin{prop}\label{charP} Let $d>0$ and $\gamma(s)\subset N^2(\rho)\subset \mathbb{R}^3$ be an arc length parametrized curve with nonconstant curvature $\kappa(s)$. Then, the function $\kappa(s)$ satisfies \eqref{fi} if and only if there exists a coordinate system such that $\gamma(s)=\left(x_1(s),x_2(s),x_3(s)\right)$ and
\begin{equation}\label{multiple}
x_1(s)=\frac{1}{\sqrt{d}}\,\dot{P}\left(\kappa(s)\right).
\end{equation}
\end{prop}
\begin{proof} The forward implication follows from standard computations involving Killing vector fields along curves. For details see, for instance, \cite{AP}.

Now, for the converse assume that an arc length parametrized curve $\gamma(s)$ satisfies \eqref{multiple}. Then, from this relation, we obtain that the curvature $\kappa(s)$ of $\gamma(s)$ locally coincides with the nonconstant curvature of a critical curve for $\mathbf{\Theta}$, whose parametrizations can be found in \cite{AP}. Therefore, by the Fundamental Theorem of Planar Curves, both curves are the same, up to rigid motions, and so $\gamma(s)$ is also critical for $\mathbf{\Theta}$ with nonconstant curvature. That is, its curvature satisfies \eqref{fi}.
\end{proof}

We are now in the right position to prove the main result of this section.

\begin{thm}\label{char} Let $M^{n-1}$ be a non-CMC rotational hypersurface in $N^n(\rho)$. Then, locally, $M^{n-1}$ is biconservative if and only if its profile curve $\gamma(s)\subset N^2(\rho)$ satisfies the Euler-Lagrange equation associated to the curvature energy functional
$$\mathbf{\Theta}_p(\gamma)=\int_\gamma \kappa^p\,,$$
where $p=(n-2)/(n+1)\in [1/4,1)$.
\end{thm}
\begin{proof} Let $M^{n-1}$ be a non-CMC biconservative rotational hypersurface, then the relation \eqref{lw} holds between the principal curvatures of $M^{n-1}$, i.e., $\mu$ and $\lambda$, where $-\mu$ represents the curvature of the profile curve.

Moreover, since $M^{n-1}$ has nonconstant mean curvature, then the curvature of its profile curve $\gamma$ is also nonconstant. Therefore, locally, by the Inverse Function Theorem we can suppose that the arc length parameter of $\gamma$, $s$, is a function of the curvature and, hence, $x_1(s)=\dot{P}(\kappa)/\sqrt{d}$ for a suitable function $\dot{P}(\kappa)$ and constant $d>0$. This proves, applying Proposition \ref{charP} that the curvature of $\gamma$ locally satisfies \eqref{fi} and so, also \eqref{EL}.

Next, we use the definitions of the principal curvatures, \eqref{k1} and \eqref{k2}, in \eqref{lw} together with $x_1(s)=\dot{P}(\kappa)/\sqrt{d}$, to obtain that
\begin{eqnarray*}
3\mu+\left(n-2\right)\lambda&=&3\frac{\dot{P}_{ss}+\rho\dot{P}}{\sqrt{d-\rho\dot{P}^2-\dot{P}_s^2}}+\left(2-n\right)\frac{\sqrt{d-\rho\dot{P}^2-\dot{P}_s^2}}{\dot{P}}\\&=&3\kappa\frac{P-\kappa\dot{P}}{\sqrt{\left(\kappa\dot{P}-P\right)^2}}+\left(2-n\right)\frac{\sqrt{\left(\kappa\dot{P}-P\right)^2}}{\dot{P}}=0\,,
\end{eqnarray*}
where in the last line we have used the equations \eqref{EL} and \eqref{fi} to simplify the expression. This represents an ODE in $P(\kappa)$ which can be explicitly solved obtaining
$$P(\kappa)=c\kappa^{(n-2)/(n+1)}\,,$$
for some constant of integration $c$. Finally, observe that any multiple of an energy gives rise to the same variational problem and so we may assume $c=1$, concluding with the statement.

Conversely, assume that the profile curve $\gamma$ of a non-CMC rotational hypersurface $M^{n-1}$ satisfies the Euler-Lagrange equation associated to $\mathbf{\Theta}_p$. Then, since its curvature cannot be constant, it also satisfies \eqref{fi} for $P(\kappa)=\kappa^p$ and a suitable $d>0$ (the fact that $d>0$ follows from the rotational invariance of the hypersurface, \cite{DC-D,P}). It follows from Proposition \ref{charP} that there exists a coordinate system in which
$$x_1(s)=\frac{p}{\sqrt{d}}\,\kappa^{p-1}\,.$$
Using this in the principal curvatures $\mu$, \eqref{k1}, and $\lambda$, \eqref{k2}, a similar computation as above shows that \eqref{lw} is satisfied, and so is \eqref{eq:biconservative-general-formula-sf}. This finishes the proof.
\end{proof}

\begin{rem} The case $n=3$ corresponds with $p=1/4$ and the above result was obtained in \cite{MP}. When $n=5$, $p=1/2$ and the associated Euler-Lagrange equation coincides with the ODE obtained in \cite{FHYZ} when studying biconservative hypersurfaces with constant scalar curvature. We point out here that among non-CMC rotational biconservative hypersurfaces, the only ones with constant scalar curvature appear when the dimension is $n=5$. Indeed, the scalar curvature of the rotational hypersurface $M^{n-1}$ of $N^n(\rho)$ is
$$R=\left(n-1\right)\left(n-2\right)\rho+4\mu^2-{\rm trace}\,S_\eta^2=\left(n-1\right)\left(n-2\right)\rho+3\mu^2-\left(n-2\right)\lambda^2\,,$$
which combined with \eqref{lw} shows that $R$ is constant if and only if $n=5$ because for non-CMC rotational biconservative hypersurfaces the principal curvatures $\mu$ and $\lambda$ are not constant.
\end{rem}

Regardless of the boundary conditions, throughout this paper we will call \emph{$p$-elastic curves}  those curves whose curvature satisfies the Euler-Lagrange equation associated to the energy $\mathbf{\Theta}_p$. They were first introduced by D. Bernoulli in a letter to L. Euler of 1738, \cite{Tr}. Since then, these planar curves immersed in a space form $N^2(\rho)$ have been widely studied in the literature. For instance, in \cite{LP} a geometric description of these curves in $\mathbb{R}^2$ was given, proving that $p$-elastic curves are of catenary-type (observe that this description in combination with Theorem \ref{char} coincides with the result of \cite{MOR} for rotational biconservative hypersurfaces of $\mathbb{R}^n$). In particular, when $p=1/2$ the energy $\mathbf{\Theta}_{1/2}$ for curves in  $\mathbb{R}^2$ was studied by W. Blaschke, in \cite{B}, obtaining  that critical curves are catenaries. Some results about $1/2$-elastic curves in $\mathbb{S}^2(\rho)$ can be found in \cite{AGM,AGP} and about $p$-elastic curves in $\mathbb{H}^2(\rho)$ in \cite{ABG}.

If the curvature of the $p$-elastic curve $\gamma$ is constant, then the rotational hypersurface $M^{n-1}$ has constant mean curvature. Thus, we are interested on critical curves with nonconstant curvature. In this case, the first integral of the Euler-Lagrange equation \eqref{fi} reads
\begin{equation}\label{fip}
\kappa_s^2=\frac{\kappa^2}{p^2\left(1-p\right)^2}\left(d\kappa^{2(1-p)}-\left[1-p\right]^2\kappa^2-\rho \,p^2\right).
\end{equation}
For fixed dimension $n\geq 3$, $p=(n-2)/(n+1)\in [1/4,1)$ is also fixed and solutions of \eqref{fip} will depend on two real parameters, namely $d>0$ and a constant of integration arising after integrating \eqref{fip}. Nevertheless, this second constant can be assumed to be zero after translating the origin of the arc length parameter, if necessary. Consequently, there is a (real) one-parameter family of curvatures $\kappa_d(s)$ solving \eqref{fip} and, hence, by the Fundamental Theorem of Planar Curves we have, up to rigid motions, a one-parameter family of $p$-elastic curves $\gamma_d$ in $N^2(\rho)$.

\begin{rem} For some particular choices of $p$, it may be possible to obtain explicitly the curvatures $\kappa_d(s)$ solutions of \eqref{fip}. One such a case corresponds to $p=1/2$ (see \cite{AGP}, where the curvatures were given in terms of trigonometric functions).
\end{rem}

\section{Closed Profile Curves}

In this section we will prove the existence of a discrete biparametric family of closed $p$-elastic curves with nonconstant curvature in the round sphere $\mathbb{S}^2(\rho)$.

\subsection{Periodic Curvature}

First observe that a necessary, but not sufficient, condition for a curve to be closed is to have periodic curvature. For $\rho\leq 0$ we will show in Proposition~\ref{periodic} that $p$-elastic curves $\gamma_d$ with $d>0$ do not have periodic curvature. To the contrary, when $N^2(\rho)=\mathbb{S}^2(\rho)$ and when the curves are defined on their maximal domain, $p$-elastic curves $\gamma_d$ have always periodic curvature. To prove this, it is convenient to rewrite \eqref{fip} in terms of a new variable $u>0$, given by $\kappa^{2(1-p)}=u^3$, as
\begin{equation}\label{Q}
u_s^2=\frac{4u^2}{9p^2}\left(-\left[1-p\right]^2u^{3/(1-p)}+d u^3-\rho\, p^2\right)=\frac{4u^2}{9p^2}Q(u)\,.
\end{equation}
Note that $3/(1-p)=n+1\in\mathbb{N}$ and so $Q(u)$ is a polynomial of degree $n+1\geq 4$. Analysing the positive roots of this polynomial we conclude with the following result.

\begin{prop}\label{periodic} Let $\gamma_d$ with $d>0$ be a $p$-elastic curve with nonconstant curvature in $N^2(\rho)$ defined on its maximal domain. Then, the curvature of $\gamma_d$ is a periodic function if and only if $N^2(\rho)=\mathbb{S}^2(\rho)$.
\end{prop}
\begin{proof} Assume that $\gamma_d$ is a $p$-elastic curve in $N^2(\rho)$ with nonconstant curvature and $d>0$. It then follows that equation \eqref{Q} must be satisfied. To simplify the notation we call $x=\sqrt{u}>0$ and $y=x_s$, so that \eqref{Q} simplifies to
\begin{equation}\label{xy}
y^2=\frac{x^2}{9p^2}Q(x^2)\,.
\end{equation}
This is an algebraic curve representing the orbit of the differential equation \eqref{Q} in the phase plane. Moreover, applying the standard square root method of algebraic geometry, if this curve is closed then the polynomial $Q(x^2)$ has at least two positive roots.

It follows that since $d>0$, if $\rho\leq 0$ there is only one change of sign between the first and second coefficients of $Q(x^2)$ (see \eqref{Q} for the definition of $Q(u)$) and, hence, we conclude from Descartes' rule of signs that in these cases $Q(x^2)$ has only one positive root. That is the algebraic curve described by \eqref{xy} cannot be closed. Next, for $\rho>0$, a similar argument shows that $Q(x^2)$ has at most two positive roots. However, it is easy to check that $Q(x^2)$ has only one positive critical point at
$$x_*=\left(\frac{d}{1-p}\right)^{\frac{1-p}{6p}}\,,$$
which is a local maximum. Consequently, whenever $Q(x_*^2)>0$ we will have exactly two positive roots. After some manipulations we see that this happens precisely whenever
\begin{equation}\label{d*}
d>d_*=\rho^{\,p}\, p^{\,p}\left(1-p\right)^{1-p}\,.
\end{equation}

Finally, assume that $\rho>0$ and $d>d_*$. In this setting, the curve $C(s)=\left(x(s),y(s)\right)$ is included in the trace of the closed regular curve described by \eqref{xy} and it can be thought as a bounded integral curve of the smooth vector field
$$X(x,y)=\left(y,\frac{x}{9p^2}\left[Q(x^2)+x^2Q'(x^2)\right]\right),$$
defined in $\{(x,y)\in\mathbb{R}^2\,\lvert\,x>0\}$. We mention here that $x=\sqrt{u}>0$ and, hence, the domain of $X$ cannot be extended to the whole plane. This implies that $C(s)$ is smooth and defined on the whole $\mathbb{R}$ and so are the associated curves $\gamma_d$. In other words, their maximal domain is the whole real line $\mathbb{R}$. Furthermore, since the vector field $X$ has no zeros along the curve $C(s)$, we conclude from the Poincar\'e-Bendixson Theorem that $C(s)$ is a periodic curve. Therefore, the nonconstant curvature of $\gamma_d$ (equivalently, the associated $u(s)$) is a periodic function. This finishes the proof.
\end{proof}

As a consequence of Proposition \ref{periodic}, in what follows we will assume that $N^2(\rho)=\mathbb{S}^2(\rho)$. We will also assume from now on that all our curves are defined on its maximal domain, i.e., on the whole real line $\mathbb{R}$. If $\gamma_d$ is a $p$-elastic curve in $\mathbb{S}^2(\rho)$ with nonconstant curvature, then $d>d_*$ and \eqref{fi} holds. Thus, from Proposition \ref{charP} we have that there exists a coordinate system such that the first coordinate of $\gamma_d$ is \eqref{multiple} for $P(\kappa)=\kappa^p$. Combining this with the fact that $\gamma_d(s)$ is an arc length parametrized curve lying on the round sphere of radius $1/\rho$, $\mathbb{S}^2(\rho)\subset\mathbb{R}^3$, and the definition of the variable $u(s)$, we conclude that
\begin{equation}\label{param}
\gamma_d(s)=\frac{1}{\sqrt{\rho\,d\,u^3(s)\,}}\left(\sqrt{\rho\,}\,p,\sqrt{du^3(s)-\rho\,p^2}\sin\psi(s),\sqrt{du^3(s)-\rho\,p^2}\cos\psi(s)\right),
\end{equation}
is a parametrization of $\gamma_d(s)$. Here, the function $\psi(s)$ is given by
\begin{equation}\label{psi}
\psi(s)=\left(1-p\right)\sqrt{\rho\,d}\int \frac{u^{3(2-p)/(2(1-p))}}{du^3-\rho\,p^2}\,ds\,,
\end{equation}
and so, after a change of variable involving \eqref{Q}, the parametrization \eqref{param} is given in terms of just one quadrature.

From this parametrization we can geometrically describe $p$-elastic curves in $\mathbb{S}^2(\rho)$. Since these curves have periodic curvature, it is enough to describe them just for one period of the curvature. The complete curves are constructed by smoothly gluing congruent copies of the part covered in one period of the curvature (see Figures \ref{Figure1} and \ref{Figure2}).

We first notice that the first coordinate of $\gamma_d(s)$ is always positive, which means that $\gamma_d$ never cuts the equator $x_1=0$ and is always contained on the half-sphere $\{(x_1,x_2,x_3)\in\mathbb{S}^2(\rho)\,\lvert\,x_1>0\}$. Next, we focus on the tangent vector field to $\gamma_d(s)$. It is clear that $x_1'(s)=-3pu'(s)u^{-5/2}(s)/(2\sqrt{d})$, and as a consequence it vanishes at the maximum and minimum curvatures (equivalently, maximum and minimum values for $u(s)$) of $\gamma_d$. At those points, the $p$-elastic curve meets tangentially two parallels, respectively. Moreover, since the curvature increases from the minimum to the maximum, $\gamma_d$ is bounded between these parallels. Indeed, since at those points (see \eqref{Q})
\begin{equation}\label{pole}
du^3-\rho\,p^2=(1-p)^2u^{3/(1-p)}>0\,,
\end{equation}
none of those parallels can be the pole of the parametrization. In other words, the curve $\gamma_d$ never passes through the pole $(1/\sqrt{\rho},0,0)$. Finally, differentiating \eqref{psi} with respect to the arc length parameter, we observe that the function $\psi(s)$ is monotonic for each half period of the curvature. Thus, the $p$-elastic curve $\gamma_d$ goes always forward and it does not cut itself in one period of its curvature, unless it gives more than one round in that period.

\subsection{Closure Conditions}

In order to obtain closure conditions we adapt some standard computations, obtaining the following result.

\begin{prop}\label{closure} Let $\gamma_d\subset\mathbb{S}^2(\rho)$ be a $p$-elastic curve with $d>d_*$ defined on its maximal domain. Then, $\gamma_d(s)$ is closed if and only if the following identity holds
\begin{equation}\label{I}
I(d)=\left(1-p\right)\sqrt{\rho\,d}\int_0^\varrho \frac{u^{3(2-p)/(2(1-p))}}{du^3-\rho\,p^2}\,ds=\frac{2\pi l}{r}\,,
\end{equation}
for some natural numbers $l$ and $r$ such that ${\rm gcd}(l,r)=1$. Here, $\varrho\equiv\varrho(d)$ represents the period of the curvature of $\gamma_d$.
\end{prop}
\begin{proof} Let $\gamma_d$, $d>d_*$, be a $p$-elastic curve defined on its maximal domain. Then, from Proposition \ref{periodic}, the curvature of $\gamma_d$ is a periodic function, and we denote by $\varrho\equiv\varrho(d)$ its period. It then follows from the parametrization \eqref{param}, that $\gamma_d$ is closed if and only if the integral $\psi(s)$, \eqref{psi}, along a natural multiple of $\varrho$ is a natural multiple of $2\pi$, i.e., assuming without loss of generality that the origin of the arc length parameter is $s=0$, if and only if
$$\psi(r\varrho)=(1-p)\sqrt{\rho\,d}\int_0^{r\varrho}\frac{u^{3(2-p)/(2(1-p))}}{du^3-\rho\,p^2}\,ds=2\pi l\,,$$
for suitable natural numbers $l$ and $r$, which may be considered co-primes. 

Finally, from the periodicity of $u(s)$ we get that $\psi(r\varrho)=r\psi(\varrho)$ and, hence, the result follows.
\end{proof}

We point out here that the numbers $l$ and $r$ have a geometric meaning. In fact, the number of times the $p$-elastic curve winds around the pole of the parametrization $(1/\sqrt{\rho},0,0)$ is represented by $l$ (indeed, $l$ is the winding number of the projection of $\gamma_d$ to the plane $\{(x_1,x_2,x_3)\in\mathbb{R}^3\,\lvert\,x_1=0\}$ around the point $(0,0)$, i.e., the projection of the rotation axis $\{(x_1,0,0)\,\lvert\,x_1\in\mathbb{R}\}$), while $r$ is the number of periods of the curvature contained in one period of the curve, i.e., the number of lobes of the $p$-elastic curve.

Let $\gamma_d$ be a $p$-elastic curve in $\mathbb{S}^2(\rho)$ with $d>d_*$. Then the curvature of $\gamma_d$ is a nonconstant periodic function and so is $u(s)$. We denote by $\alpha\equiv\alpha(d)$ (respectively, $\beta\equiv\beta(d)$) the maximum (respectively, the minimum) value of $u(s)$ associated to $\gamma_d$. If we use \eqref{Q} to make a change of variable on the integral $I(d)$, \eqref{I}, we obtain
\begin{equation}\label{INew}
I(d)=3p\left(1-p\right)\sqrt{\rho\,d}\int_\beta^\alpha \frac{u^{(4-p)/(2(1-p))}}{\left(du^3-\rho\,p^2\right)\sqrt{Q(u)}}\,du\,,
\end{equation}
where $Q(u)$ is the polynomial of degree $n+1$ introduced in \eqref{Q}. This expression can be understood as a function $I:(d_*,\infty)\subset\mathbb{R}\longrightarrow\mathbb{R}$, where $d_*$ is defined in \eqref{d*}. In order to prove the existence of closed $p$-elastic curves, we need to check that the image of $I(d)$ is not constant and, consequently, there must exist rational multiples of $2\pi$, proving the condition of Proposition \ref{closure}. Moreover, to find the possible restrictions on the parameters $l$ and $r$, we also need to analyze the limits of $I(d)$ as $d$ approaches $d_*$ and $d\to\infty$.

We prove this in the following technical lemma.

\begin{lem}\label{lemma} Let $I:(d_*,\infty)\subset\mathbb{R}\longrightarrow\mathbb{R}$ be a function defined by the integral expression \eqref{INew}. Then, $I$ is a continuous function on $d$ satisfying
$$\lim_{d\to d_*} I(d)=\sqrt{2}\,\pi\,\quad\quad\quad\text{and}\,\quad\quad\quad \lim_{d\to\infty} I(d)=\pi\,.$$
(Observe that these limits hold for every $p=(n-2)/(n+1)$ and $n\geq 3$.)
\end{lem}
\begin{proof} The continuity of the function $I(d)$ follows directly from \eqref{pole}, since the denominator of the integrand does not vanish.

We now begin proving the limit when $d\to d_*$. To compute this limit we apply Lemma 4.1 of \cite{P} (see also Corollary 4.2 of the same paper). In our setting all the conditions are satisfied and so this result shows that
$$\lim_{d\to d_*} I(d)=3p\left(1-p\right)\sqrt{\rho\,d_*}\,\frac{u_*^{(4-p)/(2(1-p))}}{\left(d_*u_*^3-\rho\,p^2\right)\sqrt{-\frac{1}{2}Q''(u_*)}}\,\pi\,,$$
where $u_*=d_*^{(1-p)/(3p)}\left(1-p\right)^{(p-1)/(3p)}$ (compare to $x_*$ of Proposition \ref{periodic}). We then simplify this expression using \eqref{d*} and obtain, after long straightforward manipulations, that $\lim_{d\to d_*}I(d)=\sqrt{2}\,\pi$.

Next, in order to compute $\lim_{d\to\infty}I(d)$, we will work in the complex plane $\mathbb{C}$, extending the integrand of $I$. We first notice that the polynomial $Q(u)$ has degree $n+1$ and exactly two positive roots, $\alpha$ and $\beta$, when $d>d_*$. Moreover, from Descartes' rule of signs, we conclude that the polynomial has exactly one negative root (which we denote by $\delta$), if $n$ is even; or zero negative roots, if $n$ is odd. Therefore, the rest of the roots lie in $\mathbb{C}\setminus\mathbb{R}$. We denote them by $\omega_j$ and $\overline{\omega}_j$, $j=1,...,\lfloor (n-1)/2\rfloor$, where the upper line denotes the complex conjugate.

Let us now define a complex function $h(z)$ by
$$h(z)=\left(-i\sqrt{-z}\,\right)^{n+2}\left(\alpha-z\right)\sqrt{\frac{z-\beta}{\alpha-z}}\,\sqrt{\sum_{j=1}^{(n-1)/2}\left(z-\omega_j\right)\left(z-\overline{\omega}_j\right)}\,,$$
if $n$ is odd, or by
$$h(z)=-i\left(-i\sqrt{-z}\right)^{n+2}\left(\alpha-z\right)\sqrt{\frac{z-\beta}{\alpha-z}}\,\sqrt{-(z-\delta)}\,\sqrt{\sum_{j=1}^{(n-2)/2}\left(z-\omega_j\right)\left(z-\overline{\omega}_j\right)}\,,$$
if $n$ is even. In both cases, the square root symbol represents the principal branch of it, i.e., $\sqrt{z}=\sqrt{r e^{i\theta}}=\sqrt{\lvert r\rvert}e^{i\theta/2}$, $\theta\in(-\pi,\pi)$. Since the Moebius transformation $(z-\beta)/(\alpha-z)$ maps the set of complex numbers $z=u+iv$ with $v=0$ and $\beta<u<\alpha$ to $\mathbb{R}^+$, the set of positive real numbers, the function $h(z)$ is well-defined and analytic far from the non positive part of the $u$-axis and the roots of the polynomial $Q(u)$. Denote by $u_o=(\rho\,p^2/d)^{1/3}$ the only positive real solution of $du^3-\rho\,p^2=0$ and by $u_1$, $\overline{u}_1$ the remaining two complex solutions. With this notation the complex function
$$f(z)=3p\left(1-p\right)\sqrt{\rho\,d}\,\frac{z^{n+2}}{\left(dz^3-\rho\,p^2\right)h(z)}\,,$$
is well-defined and analytic far from the corresponding singularities and the non positive part of the $u$-axis. Observe that the exponent of the numerator is $n+2=(4-p)/(1-p)$.

Moreover, from the definition of the functions $h(z)$ and $f(z)$, for any $u\in(\beta,\alpha)$ we have that
\begin{eqnarray*}
\lim_{\epsilon\to 0^+} f(u+i\epsilon)&=&f(u)\,,\\
\lim_{\epsilon\to 0^-} f(u+i\epsilon)&=&-f(u)\,,
\end{eqnarray*}
where $f(u)$ is, precisely, the integrand of \eqref{INew}.

Then, defining a path $\sigma$ to be a big circle enclosing all the singularities of $f$, with a cut so that the non positive part of the $u$-axis is not inside $\sigma$, $\sigma_*$ to be a small path around $\alpha$ and $\beta$, and $\sigma_{-}$ to be a little circle around each of the rest of the singularities of $f$ that lie in $\mathbb{C}\setminus\mathbb{R}^{-}$, we conclude from the analyticity of $f$ that
\begin{eqnarray*}
\int_\sigma f(z)\,dz&=&\int_{\sigma_*}f(z)\,dz+\int_{\sigma_{u_o}}f(z)\,dz+\left(\int_{\sigma_{u_1}}f(z)\,dz+\int_{\sigma_{\overline{u}_1}}f(z)\,dz\right)\\
&&+\sum_{j=1}^{\lfloor (n-1)/2\rfloor}\int_{\sigma_{\omega_j}}f(z)\,dz+\sum_{j=1}^{\lfloor (n-1)/2\rfloor}\int_{\sigma_{\overline{\omega}_j}}f(z)\,dz\,,
\end{eqnarray*}
assuming all the paths are oriented counter-clockwise. It follows from Cauchy's Integral Formula that the integrals in the second line above are all zero, and so the singularities $\omega_j$ and $\overline{\omega}_j$ for $j=1,...,\lfloor (n-1)/2\rfloor$ are removable. Moreover, we can apply once again Cauchy's Integral Formula to check that
$$\int_{\sigma_{u_1}}f(z)\,dz+\int_{\sigma_{\overline{u}_1}}f(z)\,dz=0\,.$$
Similarly, for the singularity $u_o$, we compute
$$\int_{\sigma_{u_o}}f(z)\,dz=2\pi i\left(\frac{1}{2\pi i}\int_{\sigma_{u_o}}\frac{g(z)}{z-u_o}\,dz\right)=2\pi i\, g(u_o)\,,$$
where $g(z)=(z-u_o)f(z)$. A straightforward simplification after evaluating the function $g(z)$ at $u_o$, gives that the value of above integral is $2\pi$. Therefore, we conclude with
$$\int_{\sigma_*}f(z)\,dz=\int_\sigma f(z)\,dz-2\pi\,,$$
for any $d>d_*$. Finally, along $\sigma$ the function $f(z)\to 0$ whenever $d\to\infty$, so using above limits for $f(z)$ we get when the path $\sigma_*$ tends to the interval $(\beta,\alpha)$ (which we denote by $\epsilon\to 0$) that
$$\lim_{d\to\infty} I(d)=-\frac{1}{2}\lim_{\epsilon\to 0}\left(\lim_{d\to\infty}\int_{\sigma_*}f(z)\,dz\right)=-\frac{1}{2}\left(-2\pi\right)=\pi\,.$$
This finishes the proof.
\end{proof}

\begin{rem} Above result for the case $n=3$ was shown in \cite{MP}. Moreover, for $n=5$, i.e., $p=1/2$, it was obtained in \cite{AGP} rewriting $I(d)$ in terms of elliptic integrals (this method will also work for $n=3$ since in this case $Q(u)$ is a polynomial of degree four). For $n=5$, the result was also shown in \cite{AL,P} using the same approach from Complex Analysis. In these last three papers, the integral $I(d)$ arose from studying rotational minimal surfaces in $\mathbb{S}^3(\rho)$, that is, a completely different problem.
\end{rem}

We conclude from Lemma \ref{lemma} with the following result about the existence of closed $p$-elastic curves.

\begin{thm}\label{closed} Let $l$ and $r$ be co-prime natural numbers such that $r<2l<\sqrt{2}\,r$ holds. Then, for every natural number $n\geq 3$ fixed, there exists a closed $p$-elastic curve with $p=(n-2)/(n+1)$. Moreover, this curve winds around the pole $l$ times and has $r$ lobes.
\end{thm}
\begin{proof} Let $l$ and $r$ be two natural numbers such that ${\rm gcd}(l,r)=1$ and $r<2l<\sqrt{2}\,r$. Then,
$$\pi<\frac{2\pi l}{r}<\sqrt{2}\,\pi$$
holds. Now, from Lemma \ref{lemma}, there exists a $d\equiv d_{l,r}>d_*$ such that
$$I(d_{l,r})=\frac{2\pi l}{r}\,.$$
Therefore, condition \eqref{I} of Proposition \ref{closure} is verified and, for every $n\geq 3$ we have a closed $p$-elastic curve $\gamma_{d_{l,r}}$, where $p=(n-2)/(n+1)$.

The second statement follows directly from the geometric description of $p$-elastic curves given after the parametrization \eqref{param}.
\end{proof}

For every $n\geq 3$, we have obtained in Theorem \ref{closed} a discrete biparametric family of closed $p$-elastic curves, where $p=(n-2)/(n+1)$. However, we will see that none of these curves is simple.

\begin{cor}\label{cor} Let $n\geq 3$ and $p=(n-2)/(n+1)$, then the closed $p$-elastic curve associated to the co-prime natural numbers $l$ and $r$ satisfying $r<2l<\sqrt{2}\,r$ has self-intersections.
\end{cor}
\begin{proof} Denote by $\gamma_{d_{l,r}}$ a closed $p$-elastic curve for $p=(n-2)/(n+1)$ and $n\geq 3$ fixed. For this curve $r<2l<\sqrt{2}\,r$ holds and it closes up in $l$ rounds around the pole.

We know from the description after \eqref{param} that each $p$-elastic curve is simple in each period of its curvature as long as it gives less than a complete round around the pole, which is our case since $I(d_{l,r})<\sqrt{2}\,\pi<2\,\pi$ holds (see Lemma \ref{lemma}). Consequently, the closed $p$-elastic curve $\gamma_{d_{l,r}}$ will be simple if and only if it closes up in one round, i.e., $l=1$. That is, we need the existence of a natural number $r$ such that $r<2<\sqrt{2}\,r$ holds.

However, this relation is not possible, so $\gamma_{d_{l,r}}$ cannot be simple.
\end{proof} 

From the relation $r<2l<\sqrt{2}\,r$, it is easy to check that the simplest of these closed $p$-elastic curves with $p=(n-2)/(n+1)$ corresponds to the values $l=2$ and $r=3$, i.e., this curve winds $2$ times around the pole and has $3$ lobes. Similarly, the second simplest one arises for the values $l=3$ and $r=5$. For $n=3$, these two curves were shown in \cite{MP}. Moreover, in the case $n=5$, these curves are the profile curves of the simplest minimal rotationally symmetric tori in $\mathbb{S}^3(\rho)$ (other than the Clifford torus), which are well-known (\cite{AL,AGP,AP,P}). In Figures \ref{Figure1} and \ref{Figure2}, using the parameterization \eqref{param}, we show these two curves together with the sphere $\mathbb{S}^2(\rho)$ for different values of $n\geq 4$ to see how they deform as $n$ increases. These spherical curves are shown from above, so that all of them wind around the pole $(1/\sqrt{\rho},0,0)$. For each color we represent one period of the curvature (from the maximum curvature to the minimum and back to the maximum), and so we need $r$ colors, i.e., $r$ periods of the curvature, to close the curve.

\begin{figure}[h!]
\begin{center}
\includegraphics[width=4.5cm,height=4.5cm]{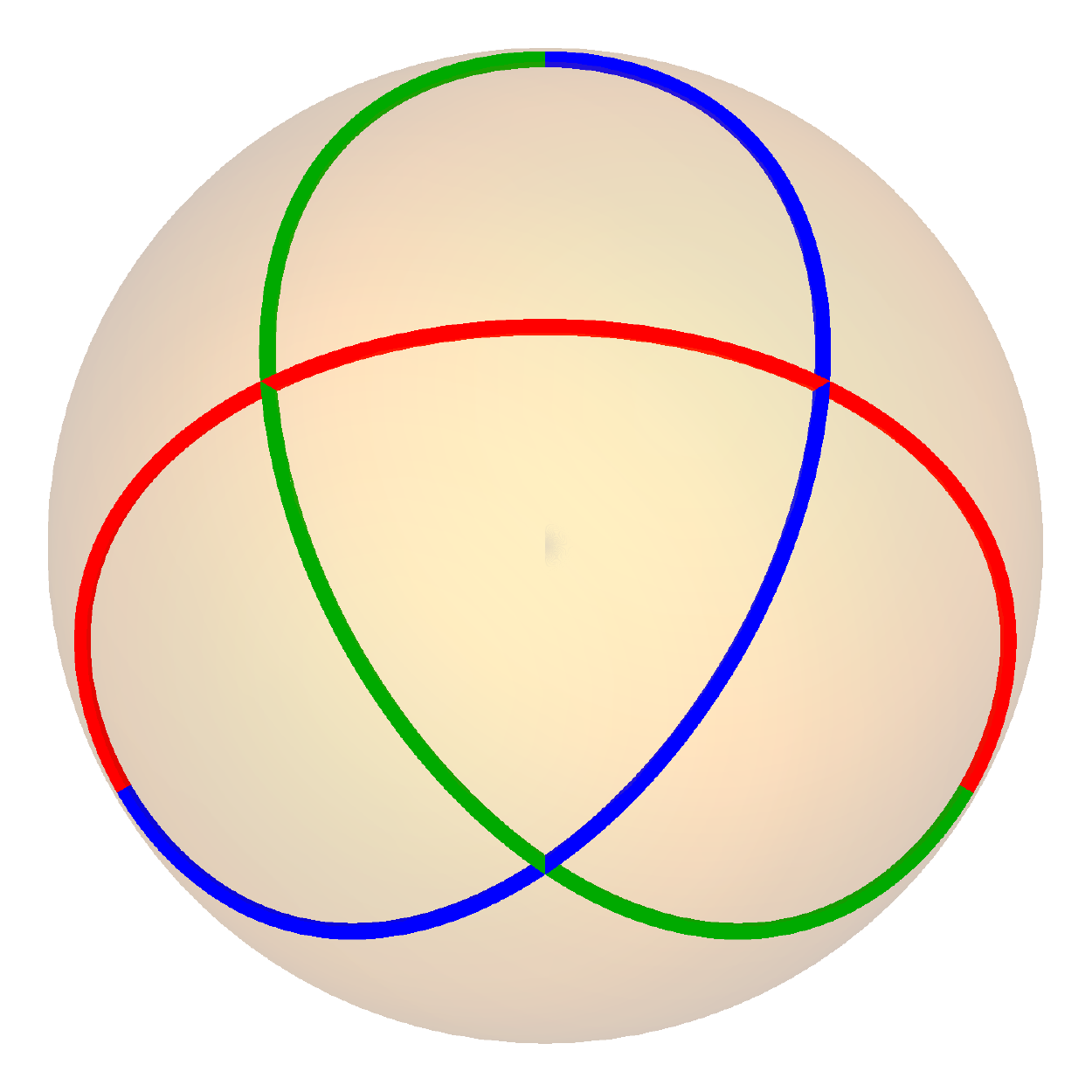}
\quad\quad
\includegraphics[width=4.5cm,height=4.5cm]{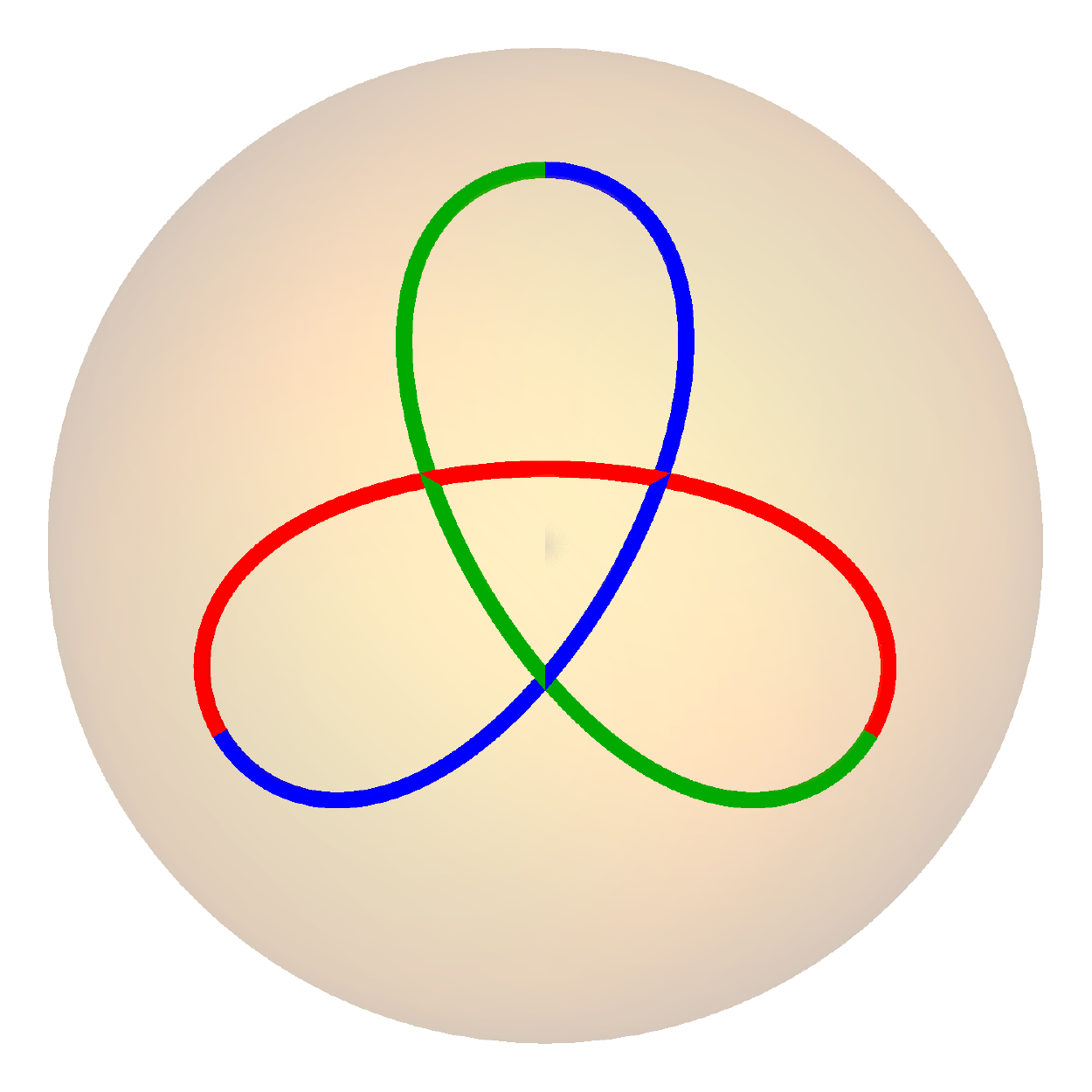}
\quad\quad
\includegraphics[width=4.5cm,height=4.5cm]{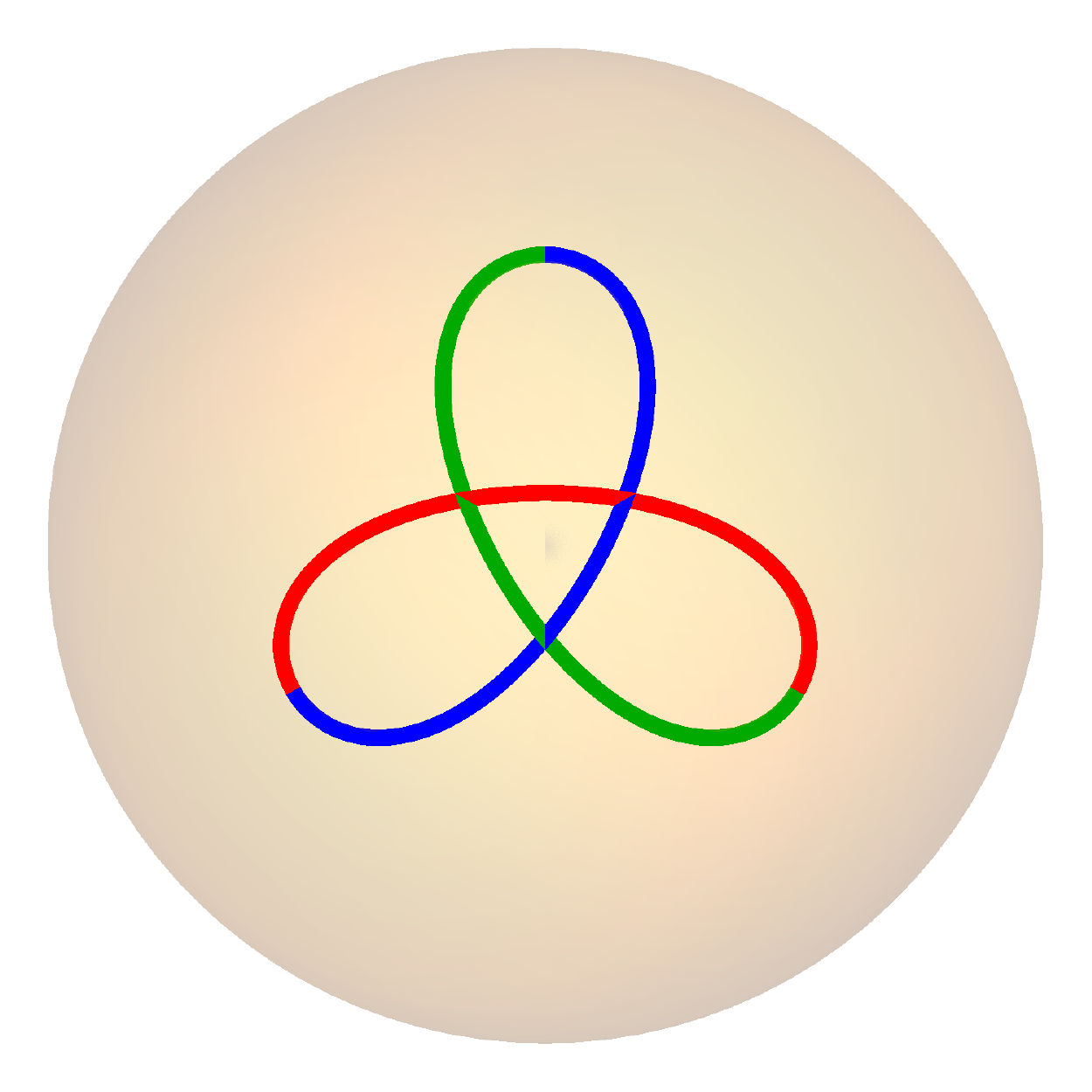}
\end{center}
\caption{Closed $p$-elastic curves for the values $l=2$ and $r=3$ for different values of $p=(n-2)/(n+1)$. From left to right: $n=4$, $n=15$ and $n=30$.}\label{Figure1}
\end{figure}

\begin{figure}[h!]
\begin{center}
\includegraphics[width=4.5cm,height=4.5cm]{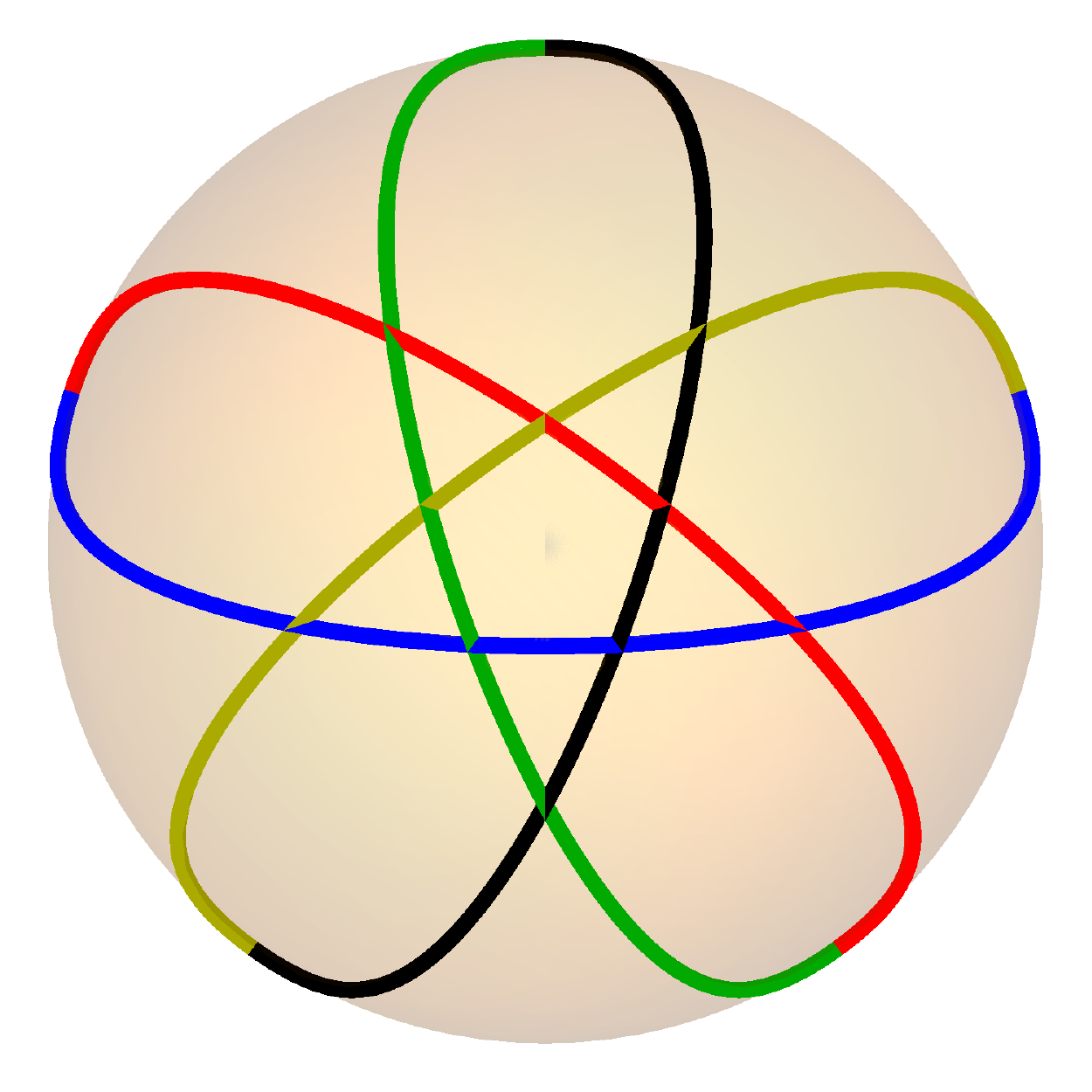}
\quad\quad
\includegraphics[width=4.5cm,height=4.5cm]{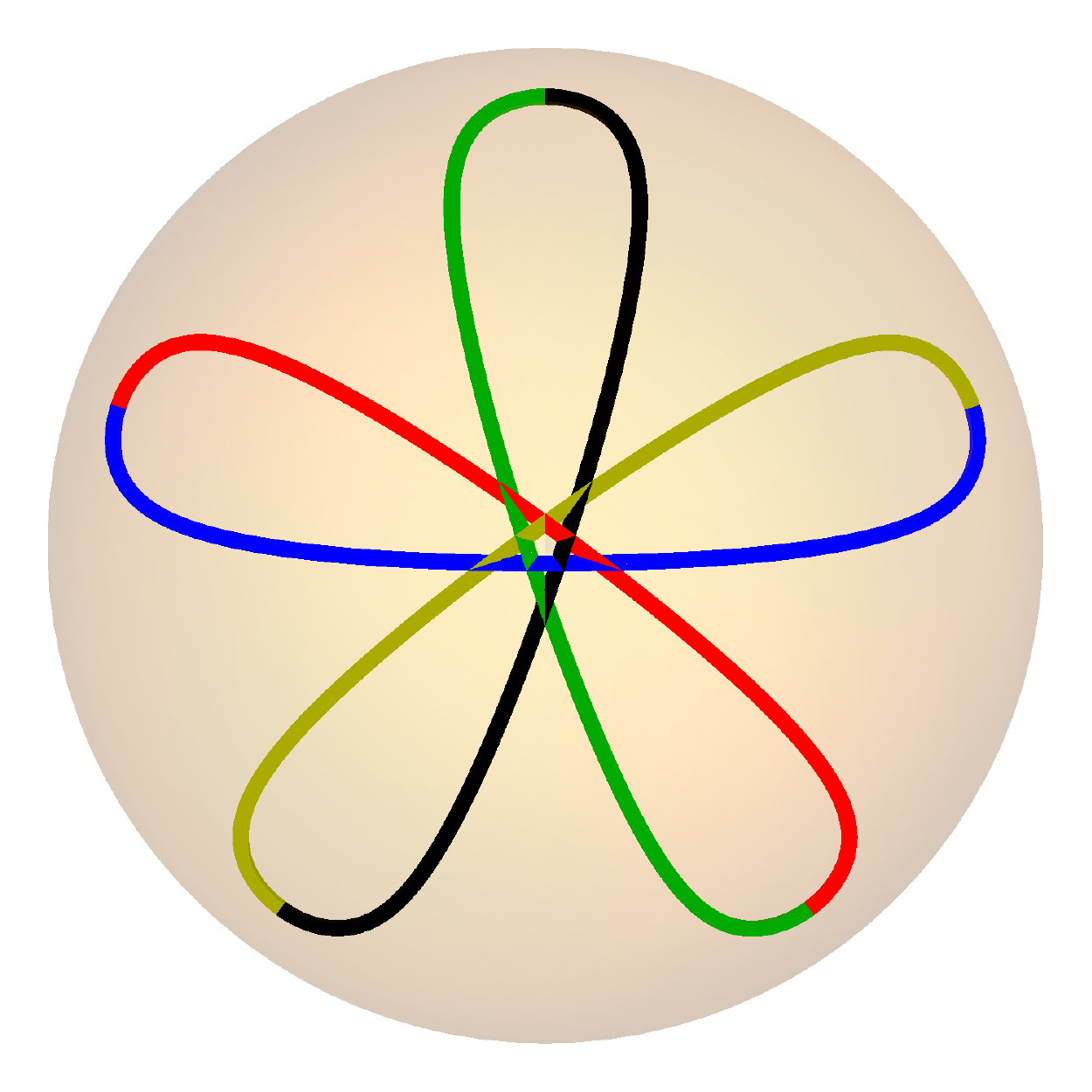}
\quad\quad
\includegraphics[width=4.5cm,height=4.5cm]{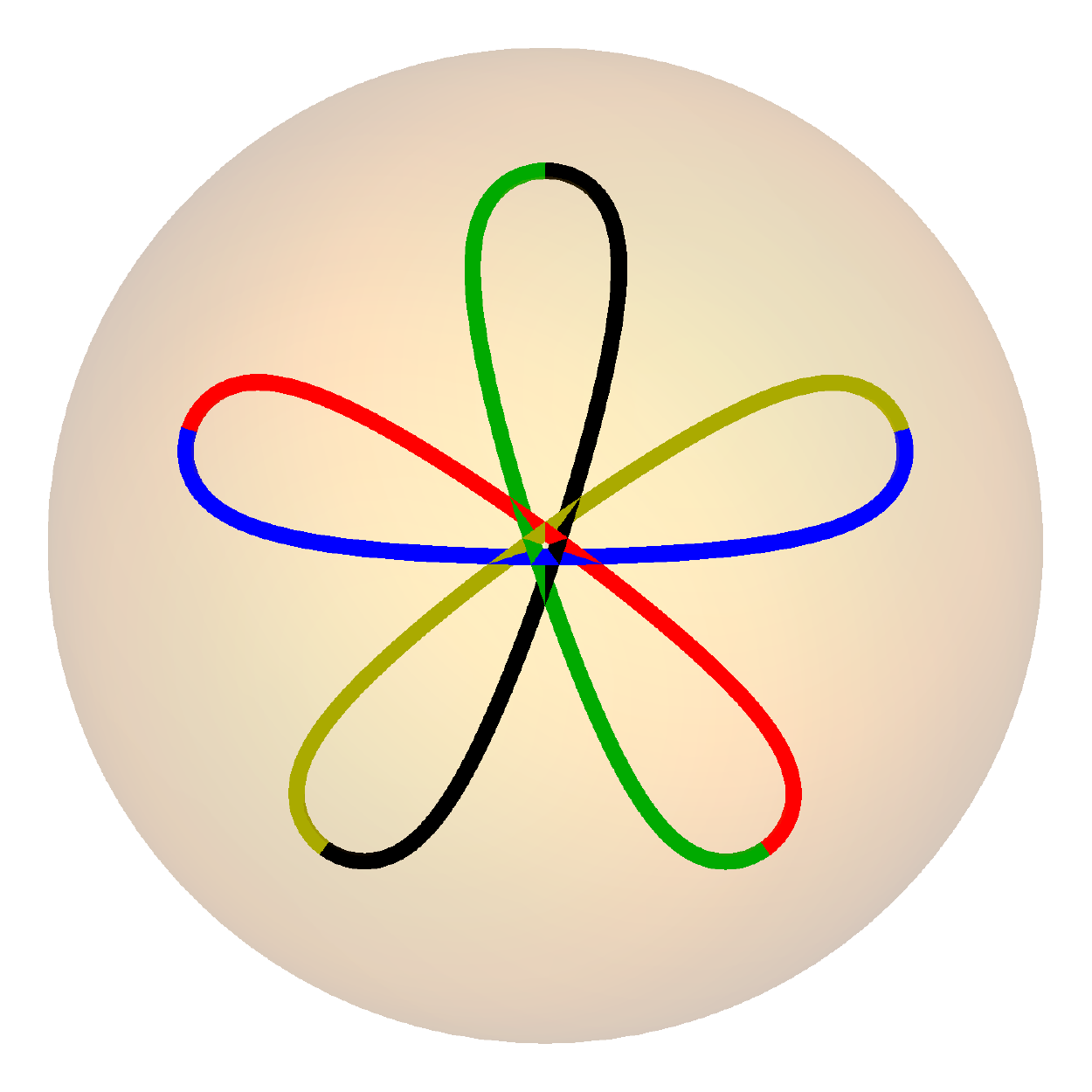}
\end{center}
\caption{Closed $p$-elastic curves for the values $l=3$ and $r=5$ for different values of $p=(n-2)/(n+1)$. From left to right: $n=4$, $n=15$ and $n=30$.}\label{Figure2}
\end{figure}

\section{Closed Rotational Biconservative Hypersurfaces}\label{5}

In this section, we use the results about $p$-elastic curves in $\mathbb{S}^2(\rho)$ to obtain rotational non-CMC biconservative hypersurfaces in spheres which are closed (i.e., compact without boundary).

Recall that from Theorem \ref{char}, locally, the profile curve $\gamma$ of a non-CMC biconservative rotational hypersurface $M^{n-1}$ is a $p$-elastic curve where $p=(n-2)/(n+1)$. Although this result is local in nature it can be employed to obtain global results too. The result of Theorem \ref{char} holds as long as the curvature of the profile curve does not attain any critical points, i.e., $\kappa'(s)\neq 0$. From the geometric description of $p$-elastic curves in $\mathbb{S}^2(\rho)$, their curvature attains the maximum and the minimum of the curvature at isolated points and so, by continuity, we can extend this result to those points. Moreover, since a rotational hypersurface $M^{n-1}$ is the evolution of the profile curve $\gamma$ under the action of $O(n-1)$, it follows that $M^{n-1}$ is closed if and only if $\gamma$ is closed. Similarly, $M^{n-1}$ will be embedded if and only if $\gamma$ is a simple curve.

In our setting we first deduce from Proposition \ref{periodic} the nonexistence of non-CMC closed biconservative rotational hypersurfaces in $N^2(\rho)$ when $\rho\leq 0$.

\begin{prop} Let $M^{n-1}$ be a non-CMC closed biconservative rotational hypersurface of a space form $N^n(\rho)$. Then, $N^n(\rho)=\mathbb{S}^n(\rho)$ is the $n$-dimensional round sphere.
\end{prop}

Furthermore, when $N^n(\rho)=\mathbb{S}^n(\rho)$, we obtain from Theorem \ref{closed} and Corollary \ref{cor} the following result.

\begin{thm}\label{claimed} For every $n\geq 3$, there exists a discrete biparametric family of closed non-CMC biconservative rotational hypersurfaces $M^{n-1}$ in the round sphere $\mathbb{S}^n(\rho)$. However, none of these hypersurfaces are embedded in $\mathbb{S}^n(\rho)$.
\end{thm}

More precisely, let $n\geq 3$ be fixed. Then, for every natural co-prime numbers $l$ and $r$ satisfying
$$r<2l<\sqrt{2}\,r$$
we have a closed non-CMC biconservative rotational hypersurface $M^{n-1}$ in $\mathbb{S}^n(\rho)$, whose profile curve winds $l$ times around the pole and has $r$ lobes.


\begin{thebibliography}{777}

\bibitem{AL} B. Andrews and H. Li. Embedded constant mean curvature tori in the three-sphere. \emph{J. Differential Geom.} \textbf{99-2} (2015), 169--189.

\bibitem{ABG} J. Arroyo, M. Barros and O. J. Garay. Closed free hyperelastic curves in the hyperbolic plane and Chen-Willmore rotational hypersurfaces. \emph{Isr. J. Math.} \textbf{138} (2003), 171--187.

\bibitem{AGM} J. Arroyo, O. J. Garay and J. J. Menc\'ia. Closed generalized elastic curves in $\mathbb{S}^2(1)$. \textit{J. Geom. Phys.} \textbf{48} (2003), 339--353.

\bibitem{AGP} J. Arroyo, O. J. Garay and A. P\'ampano. Delaunay surfaces in $\mathbb{S}^3(\rho)$. \textit{Filomat} \textbf{33-4} (2019), 1191--1200.

\bibitem{BE} P. Baird and J. Eells. A conservation law for harmonic maps. Geometry Symposium Utrecht 1980, 1-25, \textit{Lecture Notes in Mathematics  \textbf{894},} Springer, Berlin-New York, 1981.

\bibitem{BMO} A. Balmu\c s, S. Montaldo and C. Oniciuc. Biharmonic PNMC submanifolds in spheres. \emph{Ark. Mat.} \textbf{51} (2013), 197--221.

\bibitem{B} W. Blaschke. \emph{Vorlesungen uber differentialgeometrie und geometrische grundlagen von Einsteins relativitatstheorie I: Elementare differenntialgeometrie}. Springer, 1930.

\bibitem{CMOP} R. Caddeo, S. Montaldo, C. Oniciuc and P. Piu. Surfaces in three-dimensional space forms with divergence-free stress-bienergy tensor. \emph{Ann. Mat. Pura Appl.} \textbf{193} (2014), 529--550.

\bibitem{DC-D} M. Do Carmo and M. Dajczer. Rotation hypersurfaces in spaces of constant curvature. \emph{Trans. Am. Math. Soc.} \textbf{277-2} (1983), 685--709.

\bibitem{C} B-Y. Chen. \emph{Total Mean Curvature and Submanifolds of Finite Type}. Series in Pure Mathematics 1. World Scientific Publishing Co., Singapore, 1984.

\bibitem{ES} J. Eells and J. H. Sampson. Harmonic mappings of Riemannian manifolds. \textit{Amer. J. Math.} \textbf{86} (1964), 109--160.

\bibitem{FLO} D. Fetcu, E. Loubeau and C. Oniciuc. Bochner-Simons formulas and the rigidity of biharmonic submanifolds. \emph{J. Geom. Anal.} \textbf{31} (2021), 1732--1755.

\bibitem{FO} D. Fetcu and C. Oniciuc. Biharmonic and biconservative hypersurfaces in space forms. \emph{To appear in Contemp. Math.}

\bibitem{FHYZ} Y. Fu, M.-C. Hong, D. Yang and X. Zhan. Biconservative hypersurfaces with constant scalar curvature in space forms. \emph{Preprint} (2021).

\bibitem{H} D. Hilbert. Die grundlagen der physik. \textit{Math. Ann.} \textbf{92} (1924), 1--32.

\bibitem{J} G. Y. Jiang. The conservative law for 2-harmonic maps between Riemannian manifolds. \textit{Acta Math. Sinica} \textbf{30} (1987), 220--225.

\bibitem{J1} G. Y. Jiang. $2$-harmonic maps and their first and second variational formulas. \textit{Chinese Ann. Math. Ser. A} \textbf{7} (1986), 389--402.

\bibitem{LM} E. Loubeau and S. Montaldo. Biminimal immersions. \emph{Proc. Edinb. Math. Soc.} \textbf{51} (2008), 421--437.

\bibitem{LMO} E. Loubeau, S. Montaldo and C. Oniciuc. The stress-energy tensor for biharmonic maps. \textit{Math. Z.} \textbf{259} (2008), 503--524.

\bibitem{LP} R. L\'opez and A. P\'ampano. Classification of rotational surfaces in Euclidean space satisfying a linear relation between their principal curvatures. \emph{Math. Nachr.} \textbf{293} (2020), 735--753.

\bibitem{MOR} S. Montaldo, C. Oniciuc and A. Ratto. Proper biconservative immersions into the Euclidean space. \emph{Ann. Mat. Pura Appl.} \textbf{195} (2016), 403--422.

\bibitem{MP} S. Montaldo and A. P\'ampano. On the existence of closed biconservative surfaces in space forms. \emph{To appear in Commun. Anal. Geom.}

\bibitem{Simona} S. Nistor and C. Oniciuc. On the uniqueness of complete biconservative surfaces in $3$-dimensional space forms. \emph{To appear in Ann. Scuola Norm. Super. Pisa Cl. Sci.}

\bibitem{Nomizu} K. Nomizu. Characteristic roots and vectors of a differentiable family of symmetric matrices. \emph{Linear and Multilinear Algebra} \textbf{1-2} (1973), 159--162.

\bibitem{O} Y.-L. Ou. Biharmonic hypersurfaces in Riemannian manifolds. \emph{Pacific J. Math.} \textbf{248} (2010), 217--232.

\bibitem{AP} A. P\'ampano. \emph{Invariant surfaces with generalized elastic profile curves}. PhD Thesis, 2018.

\bibitem{P} O. M. Perdomo. Embedded constant mean curvature hypersurfaces on spheres. \emph{Asian J. Math.} \textbf{14} (2010), 73--108.

\bibitem{Ryan1} P. J. Ryan. Hypersurfaces with parallel Ricci tensor. \emph{Osaka J. Math}. \textbf{8} (1971), 251--259.

\bibitem{Ryan2} P. J. Ryan. Homogeneity and some curvature conditions for hypersurfaces. \emph{Tohoku Math. J.} \textbf{21-2} (1969), 363--388.

\bibitem{S} A. Sanini. Applicazioni tra varieta Riemanniane con energia critica rispetto a deformazioni di metriche. \textit{Rend. Mat.} \textbf{3} (1983), 53--63.

\bibitem{Tr} C. Truesdell, \emph{The rational mechanics of flexible or elastic bodies: 1638--1788}. Leonhard Euler, Opera Omnia, Birkhauser, 1960.

\end{thebibliography}
\end{document}